\documentclass[11pt,leqno]{amsart}
\usepackage{amsfonts, amsmath, amsthm, amssymb,xspace}
\usepackage[breaklinks=true]{hyperref}

\theoremstyle{plain}
\newtheorem{theorem}{Theorem}[section]
\newtheorem{lemma}{Lemma}[section]
\newtheorem{prop}{Proposition}[section]

\newtheorem{Conj}{Conjecture}[section]

\theoremstyle{definition}
\newtheorem{defin}{Definition}[section]

\newtheorem*{acknowledgement}{Acknowledgements}

\newcommand{\bggo}{\mathcal O}

\newcommand{\mf}[1]{\displaystyle{\mathfrak{#1}}}

\newcommand{\comment}[1]{}

\DeclareMathOperator{\spec}{\ensuremath{Spec}}
\DeclareMathOperator{\Gr}{\ensuremath{gr}}

\DeclareMathOperator{\ad}{\ensuremath{ad}}
\DeclareMathOperator{\Id}{\ensuremath{Id}}

\DeclareMathOperator{\Sym}{\ensuremath{Sym}}
\DeclareMathOperator{\Ann}{\ensuremath{Ann}}

\DeclareMathOperator{\Hom}{\ensuremath{Hom}}

\begin{document}

\title{Center of infinitesimal Cherednik algebras of $\mf{gl}_n$}
\author{ Akaki Tikaradze}
\address{The University of Toledo, Department of Mathematics, Toledo, Ohio, USA}
\email{\tt atikara@utnet.utoledo.edu}
\date{\today}
\maketitle

\begin{abstract} We show that the center of an infinitesimal Cherednik algebra of $\mf{gl}_n$
is isomorphic to the polynomial algebra of $n$ variables. As consequences of
this fact, we show that an analog of Duflo's theorem holds and all objects in the category $\bggo{}$ have finite length.
\end{abstract}
\section{Introduction}
This paper is concerned with infinitesimal Cherednik algebras. These algebras are subalgebras of continuous Cherednik algebras introduced
by Etingof, Gan and Ginzburg \cite{EGG}, as natural continuous analogs of the widely studied rational Cherednik algebras. Let us recall their definition.

\medbreak

Throughout, we will often abbreviate the Lie algebra $\mf{gl}_n$ as
$\mf{g}$, and denote by $\mf{Z}(A)$ the center of any algebra $A$.
We also fix an algebraically closed ground field $k$ (which will be assumed to have
characteristic 0 unless explicitly mentioned otherwise). Let $V$
denote the standard $n$-dimensional representation of $\mf{g}$
(the vector space of column vectors), and let $V^{*}$ denote its dual
representation.

\medbreak

The tensor algebra $T(V\oplus V^{*})$ is a representation of
$\mf{g};$ thus we may form the semi-direct product algebra
$\mf{U}\mf{g}\ltimes T(V\oplus V^{*})$ (where $\mf{U}\mf{g}$ denotes the
universal enveloping algebra of $\mf{g}$). Let $c:V\times V^{*}\to\mf{U}\mf{g}$ 
be a $\mf{g}$-invariant pairing; then we will
associate to it an algebra $H_c$ defined as the quotient of
$\mf{U}\mf{g}\ltimes T(V\oplus V^{*})$ by the following relations:
$$[v, w^{*}]=c(v, w^{*}), [v, v_1]=0=[w^{*}, w_1^{*}],$$
for all $v, v_1\in V, w^{*}, w_1^{*}\in V^{*}.$

\medbreak

It is clear that if $c=0,$ then $H_0$ (which we will denote by $H$)
is just the enveloping algebra of the semi-direct product Lie
algebra $\mf{g}\ltimes (V\oplus V^{*}).$ Let us introduce an algebra
filtration on $H_c$ by setting $\deg(v), \deg(w^*)=1$ for $v\in V,
w^{*}\in V^{*}$, and $\deg(\alpha)=0$ for $\alpha\in \mf{U}\mf{g}$. If
we pass to the associated graded algebra, we will get a surjective
homomorphism $\mf{U}\mf{g}\ltimes \Sym (V\oplus V^{*})\to \Gr H_c$
and if this map is an isomorphism, then one says that the PBW
property is satisfied and $H_c$ is an infinitesimal Cherednik
algebra (of $\mf{gl}_n$). The set of pairings 
$c:V\times V^{*}\to \mf{U}\mf{g}$ for which $H_c$ satisfies the PBW property is
described in \cite{EGG}, and it will be recalled later.

\medbreak

As the name suggests, the algebras $H_c$ are infinitesimal analogs
of rational Cherednik algebras \cite{EG}. It is an interesting problem to
develop representation theory of such algebras. The first natural
step would be to determine their center, which is the goal of this
paper. Namely, we prove that the center of an infinitesimal
Cherednik algebra of $\mf{gl}_n$ is isomorphic to the polynomial
algebra in $n$ variables. We also briefly discuss the application of
this to the analog of the BGG category $\bggo{}$ and primitive
ideals of $H_c.$ The main result of this paper is the following
(this is the generalization of our earlier result for $n=2$ \cite{T})
\begin{theorem}\label{Tqow} 
Let $H_c$ as above be an infinitesimal Cherednik algebra. Then its center
is isomorphic to the polynomial algebra in $n$ variables, and \\$\Gr \mf{Z}(H_c)=\mf{Z}(\Gr H_c).$
\end{theorem}
The proof of this theorem consists of two parts. The first part is
the computation of the center in the undeformed case, i.e., the
computation of the center of $H_{c=0}=H=\mf{U}(\mf{g}\ltimes
(V\oplus V^{*})).$ The content of this theorem was communicated to us
by M. Rais and in fact it is a special case of a more general
theorem due to D. Panyushev \cite{P}. The second part is about lifting central elements of
$H=\Gr H_c$
 to the center of $H_c.$ This will be proved by establishing the non-existence of certain outer derivations.

\acknowledgement{I am very grateful to Pavel Etingof and especially to Mustapha Rais, who has 
provided me with many results and references, including the answer in the undeformed case. 
I would also like to thank  Apoorva Khare. }

\bigbreak

\section{Center of $\mf{U}(\mf{g}\ltimes (V\oplus V^{*}))$}

We will make use of an anti-involution $j:H_c\to H_c$ defined as
follows: $j(A)=A^t$ (transpose) for $A\in \mf{g},$ and similarly
$j(v)=v^t, v\in V.$ It follows easily from the explicit formula
for $c$ given below that $j$ is well-defined. There is also another
way of seeing this: It follows from [EGG], that if we fix a basis $y_{i}$ of $V$
and its dual basis $x_i\in V^{*}$, there is an element
$c_1\in \mf{Z}(\mf{U}\mf{g})$ such that $\sum_i y_ix_i-c_1$ belongs to the center
of $H_c,$ and we may redefine $H_c$ as a quotient of 
$\mf{U}\mf{g}\ltimes T(V\oplus V^{*})$ by the relations $[\sum_i y_ix_i-c_1, a]=0$
for all $a\in H_c.$ Now it is clear that this relation is invariant under the anti-involution
$j.$

\medbreak

To formulate the result about the center of $\mf{U}\mf{g}\ltimes (V\oplus V^{*})$,
 we will have to introduce some
terminology. We will write elements of $V$ as column vectors, and
elements of $V^{*}$ as row vectors. By $y_i,$  $i=1, ..., n,$ we will
denote a column vector with 1 in the $i$-th place and 0 everywhere
else. Similarly, $x_i$ will denote the transpose (also dual) of
$y_i.$ Let $Q_1, ..., Q_n\in k[\mf{g}]$ be defined as follows:
$$ \det(t\Id-X)=\sum_{j=0}^n (-1)^j t^{n-j}Q_j(X).$$
(Note that $Q_0=1$.) For $0\leq k< n,$ let $B_{k}:\mf{g}\to \mf{g}$
be the polynomial function ($B_k$ is the gradient of $Q_{k+1}$)
defined by
$$B_{k}(X)=X^k-Q_1X^{k-1}+...+(-1)^kQ_k.$$ Now let $L$ denote the
Lie algebra $\mf{g}\ltimes (V\oplus V^{*}),$ and let $S$ denote its
Lie group. Then we may identify $L^*$ with $\mf{g}\ltimes (V^*\oplus
V)$ via the trace pairing. We will denote by $\beta_i\in
\mf{Z}(\mf{U}\mf{g})$ the image of $Q_i$ under the symmetrization
map from $k[\mf{g}]=k[\mf{g}^{*}]$ to $\mf{U}\mf{g}$ (thus $\beta_1,
..., \beta_n$ are the standard generators of
$\mf{Z}(\mf{U}(\mf{g})$)). We will introduce elements $t_i=\sum_j
[\beta_i, y_j]x_j$. Now consider the functions 
$f_k:L^{*}\to k$ ($0\leq k\leq n-1$) defined as follows: $$f_k(x,
\lambda, v)=\langle \lambda, B_{k}(x)v\rangle=\lambda B_k(x)v,$$ where $x\in
\mf{g}, \lambda\in V^{*}, v\in V.$ Now we have the following theorem
whose proof is included for completeness' sake
(the proof is also given in a recent preprint by Rais \cite{R2}, and
a more general result is proved by Panyushev [P]).

\begin{theorem} [Panyushev, Rais] The center of $\mf{U}L$ is $k[t_1,..., t_n],$ which is
also isomorphic, via the symmetrization,
map to $k[L^{*}]^L=k[L^{*}]^S=k[f_0, ..., f_{n-1}].$
\end{theorem}
\begin{proof}

At first, let us prove that every element in $k[L^{*}]^S$ lies in
$k[f_0, ..., f_{n-1}].$ The key observation is that under the
coadjoint action of $S$ on $L^{*},$ the orbit of $(x_n^{*}, Y, V) $
(we will denote this affine space by $M$)
 is dense \cite{R1},
 where $Y$ is the matrix with 1s on the subdiagonal and 0s everywhere else, and $x_n^{*}=(0, 0, ..., 1).$
  Thus, it would suffice to prove that $k[f_0, ..., f_n]\bigl\lvert_M=k[M]$, but this is immediate.

\medbreak

  Now we establish the other direction. It will be more convenient to prove it in $\mf{U}L$ itself. It is easy to see that
  the image of $f_k$ under the symmetrization is $\sum_{i=1}^n[\beta_k, y_i]x_i.$
  Now we claim that for any $\alpha\in \mf{Z}(\mf{U}\mf{g}),$ the anti-involution $j$ fixes
  the following element: $b=\sum_{i=1}^n [\alpha, y_i]x_i$ (this is true in $H_c$ for any $c$). Indeed,
we have $b=\sum y_i\alpha x_i-\alpha\sum y_ix_i$, and $j$
clearly fixes the first summand. Since $\sum y_ix_i$ commutes with
$\mf{g},$ the second summand is fixed too. So it suffices to check that
the above element commutes with $\mf{g}$ and $V$ (by
using the anti-involution). The first part follows from the following easy
lemma.
\begin{lemma} For any $\alpha\in \mf{Z}(\mf{U}\mf{g})$, the element $\sum_i [{\alpha}, y_i]x_i$
commutes with $\mf{g}$.
\end{lemma}
 So now we need to show that $[[\beta_k, y_i],v]=0$ for all $i$ and $k,$ and $v\in V$. We will check the latter equality
 in $\Sym \mf{g}\otimes \Sym V$ via the symmetrization map. It is easy to see that $[\alpha, y_i]$ maps to
 $\sum_j \frac{\partial \bar \alpha}{\partial e_{ij}}y_j$, where $\bar \alpha$ is the symmetrization of $\alpha.$
  Thus our desired equality turns into
 $$\sum_{t, p} \frac{\partial^2  Q_k}{\partial e_{ip}\partial e_{jt}} y_ty_p=0$$
for all $i, j, k.$ But this is just a consequence of properties of determinants. 
This concludes the proof of Rais's theorem.
\end{proof}
\section{Center of $H_c$}
To finish the proof of Theorem \ref{Tqow}, we need to prove that for any
$1\leq i\leq n$ there exists $c_i\in \mf{Z}(\mf{U}\mf{g})$ such that
 $\eta_i=t_i-c_i$
is in the center of $H_c.$ Notice that since such an element will necessarily commute with $\mf{g}$
and is fixed by the anti-involution $j,$ it will be sufficient to prove that this element commutes with $V.$
Notice that the endomorphism of $H_c$ defined as $\alpha\mapsto D(\alpha)=[t_i, \alpha]$
preserves $\mf{U}\mf{g}\ltimes \Sym V=\mf{U}(\mf{g}\ltimes V).$
Indeed, $$D(v)=\sum_j[\beta_i, y_j][x_j, v]+\sum_j [[\beta_
i, y_j], v]x_j,$$ but $\sum[[\beta_i, y_j], v]=0$ and $[x_j, v]\in\mf{U}\mf{g}$ so $D(v)\in \mf{U}(\mf{g}\ltimes V)$
and since
$t_i$ commutes with $\mf{g}$ we obtain that $D(\mf{U}(\mf{g}\ltimes V))\subset \mf{U}(\mf{g}\ltimes V).$
In particular, using the PBW property of $H_c,$ we see that $D$ is a $\mf{g}$-invariant derivation
of $\mf{U}(\mf{g}\ltimes V).$ Thus by Proposition \ref{TWeiwei}, it must be an inner derivation,
so there exists $c_i\in H_c$ such that
$D=\ad(c_i).$ Therefore, $t_i-c_i$ belongs to the center of $H_c.$ Since the filtration degree of 
$c_i$ is less than the filtration degree of $t_i$ (which is equal to 2), 
$c_i\in \mf{U}\mf{g}\oplus \mf{U}(V\oplus V^{*}).$ Also, since $t_i$ commutes with $\mf{U}\mf{g}$
and $\mf{U}\mf{g}V$ and $\mf{U}\mf{g}V^{*}$ have weights 1, -1 with respect to the adjoint action by $\Id\in\mf{g}$, 
we conclude that $c_i\in \mf{Z}(\mf{U}\mf{g})$.
\medbreak

Let $\phi:\mf{U}(\mf{g}\ltimes V)\to \mf{U}(\mf{g}\ltimes V)$ be a
derivation defined as follows: $\phi(\mf{sl}_n\ltimes \Sym V)=0$ and
$\phi(\Id)=1,$ where by $\Id$ we denote the identity matrix.
 Now we have the following key proposition.
\begin{prop}\label{TWeiwei} 
$\phi$ is a generating outer derivation of $\mf{U}(\mf{g}\ltimes V)$ over $k$.
\end{prop}
The proof below is due to Pavel Etingof, and it is shorter and nicer than our original proof.
\begin{proof}
        Let $G_1=G\ltimes V$ be the affine group $(G=Gl_n(k))$, and  $\mf{g}_1$ its Lie algebra, 
 so $\mf{g}_1=\mf{g}\ltimes V$. Let us identify $\mf{g}_1^{*}$ with $V^{*}\oplus \mf{g}$ via 
 the trace pairing.  The coadjoint representation of $\mf{g}_1$ has a dense orbit \cite{R1}. 
 This orbit $X$ is a principal homogeneous $G_1$-space and consists of pairs $(A,f), f\in V^*, A\in \mf{g}$, 
 such that $f, fA,..., fA^{n-1}$ are a basis of $V^*$. So the complement of the dense orbit $X$ 
 is the hypersurface defined by the equation $P(A,f)=0$, where $P\in O(\mf{g}_1^*)$ is 
 the determinant of the above vectors. The function
$P$ is a semi-invariant function with respect to the coadjoint action of $G_1,$ and 
any other such function is
a multiple of a power of $P,$ \cite{R1}. This implies that complement of $X$, which will be denoted by
$Z,$ is an irreducible variety.        

        Our goal is to show that $H^1(\mf{U}\mf{g}_1,\mf{U}\mf{g}_1)$ is 1-dimensional. This is of course the same as  $H^1(\mf{g}_1, k[\mf{g}_1^{*}])$.
        Now, consider instead $H^1(\mf{g}_1, k[\mf{g}_1^{*}][1/P])$. This is the same as 
 $H^1(\mf{g}_1, O(X))=H^1(\mf{g}_1, O(G_1)$). But for any affine algebraic group $R, H^*(\mf{r},O(R))$ 
 (where $\mf{r}$ is the Lie algebra of $R,$ and $R$ acts by left translations)
 is isomorphic to $H^*_{DR}(G),$ the De Rham cohomology of $R.$
        So $H^1(\mf{g}_1, O(G_1))$ is 1-dimensional in our case, and spanned by the derivation  $\phi$ sending $\mf{sl}(V)$ and $V$ to zero and $\Id$ to 1.

        Thus, all we need to show is that the map 
\\$H^1(\mf{g}_1, k[\mf{g}_1^{*}])\to H^1(\mf{g}_1,k[\mf{g}_1^{*}][1/P])$ is injective. Using the long exact sequence, for this it 
suffices to show that $H^0(\mf{g}_1, k[\mf{g}_1^{*}][1/P]/k[\mf{g}_1^{*}])=0.$

        Let $U_m$ be the space of elements of $k[\mf{g}_1^{*}][1/P]$ having a pole of degree at most $m$ at the surface $Z$ given by the equation $P=0.$ Then it suffices to show that $H^0(\mf{g}_1, U_m/U_{m-1})=0$ for $i\geq 1$
(as $k[\mf{g}_1^{*}][1/P]$ is the direct limit of $U_m/U_0$).
        Now, $U_m/U_{m-1}=O(Z)\otimes K^{-m}$, where $K$ is a 1-dimensional representation of $\mf{g}_1$ defined by 
its action on $P.$
        So it suffices to show that $\Hom_G(L^m,O(Z))=0$ for $m\neq 0.$

        To do so, let $h\in \Hom_G(L^m,O(Z)).$ So $h\in O(Z)$ is a function with the property that $h(gz)=\det(g)^mh(z)$ for all $z\in Z, g\in G.$ Pick a generic point $z=(A,f)$ of $Z$. In some basis, $A=diag(a_1,...,a_n)$ (with distinct $a_i$), and $f$ is a generic row vector with last coordinate 0. 
 Now consider $g_t$ a diagonal matrix, $g_t=diag(1, ..., 1, t).$ Since $A$ commutes with $g_t$ and $fg_t^{-1}=f$ we get that $h(z)=h(g_tz)=t^mh(z).$     
Thus, $h(z)=0$ for generic $z\in Z,$ hence $h=0.$


\end{proof}




\medbreak

\medbreak

 Thus we conclude that there exist elements $c_i\in \mf{Z}(\mf{U}\mf{g}), i=1,...,n$ such
 that elements $\eta_i=\sum_{j=1}^n[\beta_i, y_j]x_j-c_i, j=1, ..., n$ freely
 generate the center of $H_c.$ Since $\beta_1=\Id$ (the identity matrix)
  we have $[c_1, v]=\sum_jy_j[x_j, v],$ now since
 $[x_j, v]\in \mf{U}\mf{g},$ we get that
  $$[\beta_i, [c_1,v]]=[[\beta_i, y_j],[x_j, v]]=[\sum [\beta_i, y_j]x_j, v],$$
  thus $c_i$ is a unique element of $\mf{Z}(\mf{U}\mf{g})$ such that $[c_i, v]=[\beta_i, [c_1, v]]$ for all $v\in V.$
\section{Category $\bggo{}$}
We will denote by $\mf{h}$ the standard Cartan subalgebra of $\mf{g}$ 
consisting of the diagonal matrices.

\medbreak

We have the following natural analog of the Bernstein-Gelfand-Gelfand category $\bggo{}$ defined 
for semi-simple Lie algebras. Recall the definition from \cite{EGG}.

\begin{defin}
Category $\bggo{}$ is defined as the full subcategory of finitely
generated left $H_c-$modules on which $\mf{h}$ acts diagonalizably
and $H_{+}=\mf{U}(\mf{b}_{+}\ltimes V^{*})$ acts locally finitely, where
$\mf{b}_{+}$ denotes the standard Borel subalgebra of upper triangular
matrices.
\end{defin}

This category decomposes into a direct sum of blocks with respect to
the center of $H_c$ in the standard way. We also have the standard
definition for Verma modules \cite{EGG}.

\begin{defin}
Let $\lambda\in \mf{h}^{*}$ be a weight. The Verma module with
highest weight $\lambda$ is defined as
$M(\lambda)=H_c\otimes_{H_{+}} kv_{\lambda}$, where
$kv_{\lambda}$ is  the one-dimensional representation of $H_{+}$
corresponding to the weight ${\lambda}.$
\end{defin}

By $L(\lambda)$ we denote the unique irreducible quotient of
$M(\lambda),$ and it easily follows that every simple object in the
category $\bggo{}$ is obtained in this manner.

\medbreak

Let $M(\lambda)$ be a Verma module of highest weight $\lambda\in
\mf{h}^{*}$. Then $\eta_im=\psi(c_i)(\lambda)m$ for all $m\in M(\lambda),$ where
$\psi:\mf{Z}(\mf{U}\mf{g})\to k[\mf{h}^{*}]$ is the usual
\\Harish-Chandra homomorphism of $\mf{g}.$ Let us denote by
$k[\mf{h}^{*}]_c$ the image of $k[c_1, c_2, ..., c_n]$ under $\psi$.
Let us denote by $\Psi_c :\mf{h}^{*}=\spec k[\mf{h}^{*}]\to \spec
k[\mf{h}^{*}]_c$ the map corresponding to the restriction of $\psi$ on $k[c_1, ..., c_n]\subset \mf{Z}(\mf{U}\mf{g})$.
To summarize, we
have the following.
\begin{prop} Two Verma modules $M(\lambda), M(\mu)$ (or irreducible modules $L(\lambda), L(\mu)$) 
belong to the same block of
the category $\bggo{}$ if and only if $\Psi_c(\lambda)=\Psi_c(\mu).$
\end{prop}
\begin{prop} For $c\neq 0,$ the Harish-Chandra map is finite.

\end{prop}
\begin{proof}
We need to show that $\mf{Z}(\mf{U}\mf{g})=k[\beta_1, ..., \beta_n]$ is finite over $k[c_1, ..., c_n].$  For this we will show 
that the ring of symmetric polynomials in $x_1, ..., x_n$ is finite over the ring generated by the top terms
of images of $c_i$ under the identification $k[\mf{g}]^{G}=\mf{Z}(\mf{U}\mf{g}).$

Next, we will need to recall the precise computation of the pairing $c.$ According to the computation
from \cite{EGG}, for the given deformation parameter $b=b_0+b_1\tau+\cdots+\tau^m\in k[\tau], m>0$ one has

$$[y, x] =c(y, x)= b_0r_0(x, y) + b_1r_1(x, y) +\cdots +r_m(x, y)$$

where $r_i\in\mf{U}\mf{g}$ are the symmetrizations of the following functions on $\mf{g}:$
$$(x, (1 - tA)^{-1}y)\det(1 - tA)^{-1} = r_0(x, y)(A) + r_1(x, y)(A)t + r_2(x, y)(A)t^2 + \cdots.$$

Recall that if we set $\sum \beta_jt^{n-j}$ to be the symmetrization of $\det(t\Id-A)=\sum Q_{n-j}(A)t^j\in k[\mf{g}][t]$ (where $\mf{g}=\mf{gl}_n$ and $\beta_j\in \mf{Z}(\mf{U}\mf{g}$),
then $\sum_i[\beta_j, y_i]x_i$ is in the center of $H$ (undeformed case). 
To lift these elements to the center of $H_c,$
we want to find $c_i\in \mf{Z}(\mf{U}\mf{g})$ such that
 $$\sum_j[\beta_j, y_j][x_j, y]=[c_i, y], i=1,\cdots ,n,$$
for all $y.$ Of course, this is enough to do for $y=y_1.$ Below we will compute the top symbol of $c_i.$  
We will write these equalities  in $k[\mf{g}]\otimes V$ via the symmetrization map. In this case
if $f\in k[\mf{g}],$ then 
$[f, y_i]=\sum_j \frac {\partial f}{\partial e_{ij}}.$ 
We will consider diagonal matrices $A=diag(\lambda_1,..., \lambda_n).$
For functions appearing in the above equalities we have that $\frac{\partial }{\partial e_{ij}}=0$
 if $i\neq j$ and $\frac{\partial }{\partial e_{ij}}=\frac{\partial }{\partial \lambda_i}$ for $i=j.$
Thus, we get

 $$\sum_i \sum_t \frac {\partial b_j}{\partial e_{it}}y_t[x_i, y_1]=
 \sum_t\frac {\partial c_j}{\partial e_{1t}} y_t.$$

This gives 
$$\frac{\partial c_j}{\partial \lambda_1}=\det(1-\tau A)^{-1} (x_1, (1-\tau A)^{-1}y_1)\frac{\partial Q_j}
{\partial \lambda_1}.$$
Now multiply this equation by $t^{n-j}$ and add. Then
we want to find  $c'=\sum_{j=1}^nt^{n-j}c_j,$ which is a symmetric function in $\lambda_1, ..., \lambda_n$ with values in
$k[t][[\tau]]$, such that 
$$\frac{\partial c'}{\partial \lambda_1}=-\frac{\prod_{i=2}^n (t-\lambda_i)}{\prod_{i=1}^n(1-\tau\lambda_i)(1-\tau\lambda_1)}.$$
This element is given by the following formula:
$$c'=\frac{\det(t-A)}{(t\tau-1)\det(1-\tau A)}=\frac{\prod_{i=1}^n(t-\lambda_i)}{(t\tau-1)\prod_{i=1}^n(1-\tau\lambda_i)}.$$
 This follows from $$\frac{t-\lambda_i}{(1-\lambda_i\tau)(t\tau-1)}=
 \frac{\tau^{-1}(1-\tau\lambda_i)+t-\tau^{-1}}{(t\tau-1)(1-\tau\lambda_i)}=
 \frac{\tau^{-1}}{t\tau-1}+\frac{\tau^{-1}}{1-\tau\lambda_i}.$$
 
 To summarize, if we write
$$\frac{\prod_{i=1}^n(t-\lambda_i)}{(t\tau-1)\prod_{i=1}^n(1-\tau\lambda_i)}=\sum_i\sum_j \eta_i^jt^j\tau^i$$
with $\eta_i^j\in k[\lambda_1, ..., \lambda_n]$ a symmetric polynomial, 
then the top symbol of $c_j$ 
is equal to the top symbol of the symmetrization of $\eta_m^{n-j}.$ 
 
 It is now easy to see that $k[\lambda_1, ..., \lambda_n]$ is finite over $k[c_1, ..., c_n].$
Indeed, from the formulas above we see that $\sum_j \lambda_1^j\eta_m^j=0$ and $\eta_m^{n-i}=c_i$ for $i=1,\cdots ,n,$
and $\eta_j\in k$ for $j>n-1,$ so we are done.

\end{proof}
For example, if $c:V\times V^{*}\to k$ is the standard pairing (in which case $H_c=\mf{U}\mf{g}\ltimes Weyl(V)$)
then the Harish-Chandra map $\Psi_c$ is an isomorphism.

We have the following analog of Duflo's theorem (there is a similar result for
infinitesimal Hecke algebras of $\mf{sl}_2$ \cite{KT}).
\begin{theorem} If $c\neq 0,$ then all objects in the category 
$\bggo{}$ have finite length, and any  
primitive ideal of $H_c$  is equal to $\Ann L(\lambda)$ for some weight $\lambda\in \mf{h}^{*}.$ 

\end{theorem}
\begin{proof} This is just an easy consequence of finiteness of the analog of
the Harish-Chandra map $\Psi_c$ and of Ginzburg's generalization of Duflo's
theorem \cite{G}. Indeed, if $M$ is a simple $H_c$-module, then Schur's lemma implies that $M$ is
annihilated by some maximal ideal $m\subset \mf{Z}(H_c).$ Now let us consider the following
(non-unital) subalgebras of $H_c$: let $A_{+}$ denote the subalgebra of $H_c$ generated by
$\mf{n}_{+}\ltimes V^{*},$ and let $A_{-}$ denote the subalgebra generated by $\mf{n}_{-}\ltimes V.$
Also, let $\delta$ be equal to $-a\tau+h$ where $a$ is a sufficiently big positive number
and $h\in \mf{h}$ is a sufficiently generic element such that $\ad(h)$ has strictly positive
eigenvalues on $\mf{n}_{+}.$ Thus $\ad(\delta)$ has strictly positive (respectively negative) integer
eigenvalues on $A_{+}$ (respectively $A_{-})$. Now the triple $(A_{+}, A_{-}, \delta)$ gives what is called
a noncommutative triangular structure on $H_c,$ and since $H_c/mH_c$ is finitely generated
as $A_{+}-A_{-}$ bimodule (this follows from finiteness of $\Psi_{c}$),
Ginzburg's theorem implies that $\Ann(M)$ is the annihilator of a simple object in the category $\bggo{}$.

\end{proof}

\medbreak

\medbreak

Let us briefly discuss the characteristic $p$ case. We have
\begin{prop}
If char($k)\gg 0$, then the $p$-th powers of $V, V^{*}$ and
$a^p-{a}^{[p]}, a\in \mf{g}$ (restricted powers) belong to the center of $H_c.$
\end{prop}
\begin{proof} It is clear that for any $a\in \mf{g},$ the element 
$a^p-a^{[p]}$ belongs to the center of $H_c$ when $p>\dim V.$
Let $v\in V$. Clearly, $v^p$ commutes with $V$ and $\mf{U}\mf{g}.$
Thus it remains to show that $[v^p, w]=0$ for any $w\in V^{*}.$
Recall the well-known identity in any algebra of characteristic $p:$
$\ad(a)^p=\ad(a^p).$ We have $[v^p, w]=\ad(v)^{p-1}([v, w])$; but
$[v, w]$ is an element of $\mf{U}\mf{g}$ whose filtration degree 
with respect to the standard filtration on $\mf{U}\mf{g}$ is 
less than $p-1$. Therefore $\ad(v)^{p-1}([v, w])=0$, and we are done.
\end{proof}

Let us denote by $Z_0$ the subalgebra of $H_c$ generated by the elements in the proposition. Then it follows from the above proposition
and the PBW property of $H_c$ that $Z_0\subset Z(H_c)$ and $H_c$ is a free $Z_0$-module of rank $p^{\dim(H_c)}=p^{n^2+2n}.$
There is a similar result for rational Cherednik algebras,
in which case $\Sym (V^p)^\Gamma, \Sym ((V^{*})^p)^{\Gamma}\subset\mf{Z}(H_c)$, where $\Gamma$ is 
the corresponding reflection group.

It is known that for Cherednik algebras in positive characteristic (as long as the characteristic
does not divide the order of the reflection group), the smooth and Azumaya loci coincide \cite{BC}.
This leads to the following.
\begin{Conj} For $p\gg 0,$ we have that $\Gr \mf{Z}(H_c)=\mf{Z}(\Gr H_c)$ and the smooth and the Azumaya
loci of $\mf{Z}(H_c)$ coincide.

\end{Conj}
Positive evidence for this conjecture is given by $\mf{U}\mf{sl}_{n+1},$ which can be
realized as an infinitesimal Cherednik algebra of $\mf{gl}_n.$ Next, we will verify the above
conjecture for $n=1.$

Let us fix the notation: $H_c$ denotes an algebra generated by elements $h, e, f$ with relations 
$[h, e]=e, [h, f]=-f, [e, f]=c(h),$ where $c$ is a polynomial in one variable over the ground field $k,$
which will be assumed to be algebraically closed of characteristic $p.$ Also, we will assume that $\deg c<p-1$
This is an infinitesimal Cherednik algebra for $\mf{gl}_1.$ This algebra was first considered by Smith \cite{S}.
The follwoing result also follows from \cite{BC}.
\begin{prop} We have $\Gr \mf{Z}(H_c)=\mf{Z} (\Gr (H_c)).$ Moreover, 
the smooth and Azumaya loci
of $H_c$ coincide, and the PI-degree of $H_c$ is equal to $p.$
\end{prop}
\begin{proof}
We have the analog of the Casimir $\Delta_c=ef+z(h)$ for some $z(h)\in k[h],$ and
it is a central element in $H_c$ \cite{S}. It is easy
to check that the elements $e^p, f^p, h^p-h, \Delta_c$ are central. It is clear
that the center of $\Gr H_c=H_0$ is generated by $e^p, f^p, h^p-h, ef.$ Thus we see
that $\Gr \mf{Z}(H_c)=\mf{Z}\Gr (H_c).$ It is also clear that the degree of $\Delta_c$
over $Z_0=k[f^p, e^p, h^p-h]$ is equal to $p$; thus the PI-degree of $H_c$
is $p.$ 

Now we claim that if for a given character of the center
$\chi :\mf{Z}(H_c)\to k$ we have $\chi(f^p)\neq 0,$ or $\chi(e^p)\neq 0,$
 then any irreducible module $M$ corresponding
to ${\chi}$ has dimension $\geq p.$ Indeed, let us assume that $\chi(f^p)\neq 0.$ Let us
pick an arbitrary nonzero weight
vector $v\in V$ with respect to $h;$ thus $hv=\lambda v$ for some $\lambda\in k.$ 
Then the elements $v, fv, ..., f^{p-1}v$ are nonzero and have different weights
with respect to $h;$ thus they are linearly independent and $\dim V\geq p.$ 
But since $p$
is equal to the largest dimension of a simple $H_c$-module, $\dim V=p.$ Thus characters of the 
above type
lie in the Azumaya locus. This implies that the codimension of the complement
of the Azumaya locus has codimension $\geq 2.$ By a result of Brown-Goodearl \cite{BG}, 
this implies
that the smooth and Azumaya loci coincide.
 
\end{proof}

\vspace{5mm}

\end{document}